\documentclass[twocolumn]{jsiamletters}
\group{
Continuous Optimization
}
\affiliation{Graduate School of Information Science and Technology, The University of Tokyo}{7-3-1, Hongo, Bunkyo-ku, Tokyo, 113-8656, Japan}
\affiliation{Department of Mathematical and Computing Science, School of Computing, Institute of Science Tokyo}{2-12-1, Ookayama, Meguro-ku, Tokyo, 152-8550, Japan}
\authorinfo{Atsushi Tabei}{1}{atsushi.tabei.research@gmail.com}
\authorinfo{Ken'ichiro Tanaka}{2*}{kenichiro@comp.isct.ac.jp}
\email{kenichiro@comp.isct.ac.jp}
\title{Convergence analysis of dynamical systems for optimization by an improved Lyapunov framework}
\abstract{
We study the convergence analysis of continuous-time dynamical systems associated with optimization methods for strongly convex functions. 
Recent works have proposed systematic constructions of Lyapunov functions for such analysis, 
while also revealing limitations of the Lyapunov analysis.
Aujol--Dossal--Rondepierre (2023)
have proposed a technique to address this issue by 
reorganizing Lyapunov functions so as to evaluate a quantity 
$f(x(t)) - f_* - g(t)\|x(t)-x_*\|^2$
rather than 
$f(x(t)) - f_*$. 
By combining this technique with our computer-assisted framework to discover Lyapunov functions, 
we develop an improved method that reproduces an existing convergence rate or yields better rates than previous studies.
}
\theoremstyle{plain}
\newtheorem{thm}{Theorem}
\newtheorem*{thm*}{Theorem}
\newtheorem{lem}{Lemma}
\newtheorem{rem}{Remark}

\theoremstyle{definition}
\newtheorem{dfn}{Definition}

\allowdisplaybreaks[4]

\begin{document}

\maketitle

\section{Introduction}

We consider the problem of minimizing a function 
$f : \mathbb{R}^n \to \mathbb{R}$ whose minimum value is $f_{\ast}$
and focus on the analysis of optimization methods through their continuous-time counterparts.

It is well known that many optimization algorithms
admit interpretations as continuous-time dynamical systems described by ordinary differential equations (ODEs). 
For example, 
the gradient descent method that updates the iterates by
\begin{equation}
    \label{eq:GD}
    x_{k+1} = x_k - \tau \nabla f(x_k)
    \qquad
    (\tau > 0)
\end{equation}
can be associated with the ODE
\begin{equation}
\label{eq:grad_flow}
    \dot{x}(t) = -\nabla f(x(t)).
\end{equation}
Furthermore, its accelerated variants also admit continuous-time counterparts.
This correspondence is useful for the design and analysis of optimization methods,
as it enables the study of convergence behavior using dynamical systems theory.
Consequently, substantial work has focused on the convergence rates of
$f(x(t)) - f_{\ast}$ along trajectories $x(t)$ of such systems~\cite{SuBoydCandes2016, Wilson2021}.

A central technique in this line of research is the use of Lyapunov functions.

\begin{dfn}[Lyapunov function]
\label{Lyapunov_def}
For a continuous dynamical system defined by an ODE, a function
$\mathcal{E} : \mathbb{R}_{\ge 0} \to \mathbb{R}$
is called a Lyapunov function if it is monotonically non-increasing along the trajectory of the system and bounded from below.
\end{dfn}

A typical method to derive convergence rates using Lyapunov functions 
is given as follows (cf.~\cite{SuBoydCandes2016, Wilson2021}). 

\begin{lem}
\label{lem:Lyapunov_basic}
Suppose that a Lyapunov function $\mathcal{E}$ satisfies
\begin{equation}
    \mathcal{E}(t) = \mathcal{F}(t) + e^{\gamma(t)}(f(x(t)) - f_*),
\quad
\mathcal{F}(t) \ge 0
\end{equation}
for some $\mathcal{F}(t)$ and $\gamma(t)$. 
Then the system satisfies
\begin{equation}
    \label{eq:convergence_rate}
    f(x(t)) - f_* = \mathrm{O}(e^{-\gamma(t)}).
\end{equation}
\end{lem}

\begin{proof}
Since $\mathcal{E}$ is non-increasing, we have
\begin{equation}
    e^{\gamma(t)}(f(x(t)) - f_*) \leq \mathcal{E}(t) \leq \mathcal{E}(0),
\end{equation}
which immediately implies \eqref{eq:convergence_rate}.
\end{proof}

While powerful, this method requires non-trivial construction of Lyapunov functions, traditionally relying on problem-specific insight.
Suh--Roh--Ryu \cite{SuhRohRyu2022} recently proposed a systematic approach based on conserved quantities,
but their procedure involves some arbitrariness in selecting candidate structures.

In our previous work \cite{TabeiTanaka2025},
we eliminated this arbitrariness through a computer-assisted exhaustive search for Lyapunov functions and dynamical systems.
Combined with Lemma~\ref{lem:Lyapunov_basic},
this approach reproduces known results and, in some cases, improves convergence rates for several continuous-time systems.
In this work, we directly estimate the objective gap $f(x(t)) - f_*$ using the constructed Lyapunov functions.

However, our previous framework also reveals a limitation. 
For the system corresponding to Nesterov's accelerated gradient method (NAG) 
\begin{align}
\ddot{x} + \frac{r}{t}\dot{x} + \nabla f(x) = 0
\quad
(r > 0),
\label{eq:NAG}
\end{align}
it just gives a worse convergence rate than 
that obtained in the work of Su--Boyd--Cand\`es {\cite[\S 4.3 Theorem 8]{SuBoydCandes2016}}
in the case that $r > 3$ and
$f$ is strongly convex (Definition~\ref{dfn:strongly_convex} below). 
In this case, 
they employ a function $\mathcal{E}(t)$
that does not satisfy the conditions in Definition~\ref{Lyapunov_def}
and derive the rate via an inductive argument.

This observation indicates that the convergence analysis is not restricted to the framework based on Lemma~\ref{lem:Lyapunov_basic}. 
Indeed, since the ultimate goal is to establish the boundedness of 
$\mathrm{e}^{\gamma(t)} (f(x(t)) - f_{\ast})$, 
the specific form of the Lyapunov function in Lemma~\ref{lem:Lyapunov_basic} is not necessarily essential. 
In this direction, 
Aujol--Dossal--Rondepierre
\cite{AujolDossalRondepierre2023}
have proposed an alternative approach employing more general Lyapunov functions for evaluating
\begin{align}
f(x(t)) - f_* - g(t)\|x(t)-x_*\|^2
\label{eq:mod_quant}
\end{align}
for some $g(t)$ instead of $f(x(t)) - f_*$, 
where $x_*$ is a minimizer of $f$. 
They derive convergence rates of $f(x(t)) - f_*$ by 
evaluating the quantity in \eqref{eq:mod_quant} and $g(t)\|x(t)-x_*\|^2$. 
Their method gives flexibility to the Lyapunov framework 
and increases the likelihood of finding better rates.

In this paper, we observe that our framework of \cite{TabeiTanaka2025} can be naturally combined with this technique, 
in the sense that Lyapunov functions for the modified quantity in~\eqref{eq:mod_quant} can also be discovered in the computer-assisted manner. 
As a result, 
we find two significant results that are better than those of \cite{TabeiTanaka2025}. 
Specifically, 
for the system in~\eqref{eq:NAG}, 
we reproduce the convergence rate in \cite{SuBoydCandes2016} by a simpler argument
with a Lyapunov function given by the framework of \cite{TabeiTanaka2025}. 
Furthermore, 
for the system of a generalized NAG \cite{ChengLiuShang2025}
\begin{align}
\ddot{x} + \frac{r}{t^\alpha}\dot{x} + \nabla f(x) = 0
\quad
(r > 0,\ 0 < \alpha < 1),
\label{eq:gen_NAG_alpha}
\end{align}
we establish a convergence rate strictly better than that derived in \cite{TabeiTanaka2025}.
Throughout this paper, we assume that the objective function $f$ is strongly convex.

\begin{dfn}[$\mu$-strongly convex]
\label{dfn:strongly_convex}
For $\mu > 0$, a differentiable function $f$ is said to be $\mu$-strongly convex if
\begin{equation}
    f(y) \ge f(x) + \langle \nabla f(x), y-x \rangle + \frac{\mu}{2}\|y-x\|^2
\end{equation}
holds for all $x,y \in \mathbb{R}^n$.
\end{dfn}

\noindent
In addition, 
\(x:(0,\infty)\to\mathbb R^n\) is assumed to be a \(\mathrm{C}^2\) solution of the corresponding dynamical system.

\section{Dynamical system $\ddot{x} + \frac{r}{t}\dot{x} + \nabla f = 0$}
We consider the continuous dynamical system in~\eqref{eq:NAG}. 
We assume $r>0$.
This system corresponds to the optimization method known as Nesterov's accelerated gradient method (NAG) for convex functions. 
In {\cite[\S 4.3 Theorem 8]{SuBoydCandes2016}}, for $r > 3$, the convergence rate 
\begin{equation}
    f - f_* = \mathrm{O} \left( t^{-\frac{2}{3}r} \right)
    \label{eq:m_2_over_3_rate}
\end{equation}
is given through inductive method
with a certain function $\mathcal{E}(t)$. 
The authors \cite{TabeiTanaka2025} tried to derive the same rate via Lyapunov analysis. 
However, they only got
\begin{equation}
    f - f_* = \mathrm{O} \left( t^{-\frac{1}{2}r-\frac{1}{2}} \right),
\end{equation}
which is strictly worse for $r>3$. 
In this paper, we reproduce the rate in~\eqref{eq:m_2_over_3_rate} via Lyapunov analysis.

To this end, 
we perform our computer-assisted procedure in \cite{TabeiTanaka2025}
to list candidates of Lyapunov functions. 
Among them, 
we choose the function
\begin{align}
        \mathcal{E}(t) &= t^{\frac{2}{3}r} \left( f - f_* - \frac{r^2 - 3r}{9t^2} \|x-x_*\|^2 \right) 
        \notag \\
        &\phantom{=} + \frac{1}{2} t^{\frac{2}{3}r} \left\| \dot{x} + \frac{2r}{3t} (x-x_*) \right\|^2.
        \label{eq:Lyap_NAG}
\end{align}
Then we can show that $\dot{\mathcal{E}}(t) \leq 0$ for $t \geq T$ with
\begin{equation}
    T = \sqrt{\frac{2r^2}{9\mu}}.
        \label{eq:def_T_NAG}
\end{equation}

\begin{lem}
    \label{lem:decrease_NAG}
    Suppose that $f$ is $\mu$-strongly convex and $x(t)$ satisfies
    \eqref{eq:NAG}.
    Then $\mathcal{E}(t)$ in~\eqref{eq:Lyap_NAG} is non-increasing for $t \geq T$, 
        where $T$ is defined by \eqref{eq:def_T_NAG}. 
\end{lem}

\begin{proof}
Because of \eqref{eq:NAG}, 
    simple algebra gives
\begin{align}
            \dot{\mathcal{E}}(t) &= - t^{\frac{2}{3}r} \biggl( \frac{2r}{3t}(f_* - f - \langle \nabla f, x_* - x \rangle) 
            \notag \\
            &\phantom{=- t^{\frac{2}{3}r} \biggl( \frac{2r}{3t}\quad}- \frac{2r^3 - 18r}{27t^3} \|x-x_*\|^2 \biggr).
\end{align}
    Since $f$ is $\mu$-strongly convex, the inequality
    \begin{equation}
        \dot{\mathcal{E}}(t) \leq - t^{\frac{2}{3}r} \cdot \frac{r(9\mu t^2 - 2r^2) + 18r}{27t^3} \|x-x_*\|^2
    \end{equation}
    holds. 
    Since
    $t^2 \geq \frac{2r^2}{9\mu}$ holds 
    for $t \geq T$,
    we have
    \begin{equation}
    \dot{\mathcal{E}}(t) \leq - t^{\frac{2}{3}r} \cdot \frac{2r}{3t^3} \|x-x_*\|^2 \leq 0.
    \end{equation}
    Therefore, $\mathcal{E}(t)$ is non-increasing for $t \geq T$.
\end{proof}

\begin{lem}
    \label{lem:order_main_NAG}
    Suppose that $f$ is $\mu$-strongly convex and $x(t)$ satisfies
    \eqref{eq:NAG}.
    Then, 
    the following holds:
    \begin{equation}
        f - f_* - \frac{r^2 - 3r}{9t^2} \|x-x_*\|^2 = \mathrm{O} \left( t^{-\frac{2}{3}r} \right).
        \label{eq:order_main_NAG}
    \end{equation}
\end{lem}

\begin{proof}
The left-hand side of \eqref{eq:order_main_NAG} is non-negative for $t \geq T$, 
where $T$ is defined by \eqref{eq:def_T_NAG}.
If $0 < r \leq 3$, it follows from $f - f_{\ast} \geq 0$. 
If $r > 3$, 
we use the inequality $9t^2 \geq \frac{2r^2}{\mu}$ for $t \geq T$. 
Since $f$ is $\mu$-strongly convex, 
\begin{align}
            & f - f_* - \frac{r^2 - 3r}{9t^2} \|x-x_*\|^2
            \notag \\
            &\geq \left(f - f_* - \frac{\mu}{2}\|x-x_*\|^2\right) + \frac{r}{3t^2}\|x-x_*\|^2 \geq 0
            \label{eq:non_negative_ineq}
\end{align}
holds. 
Then, by Lemma \ref{lem:decrease_NAG}, 
    we have \eqref{eq:order_main_NAG} from the same discussion as Lemma \ref{lem:Lyapunov_basic}
    for $t \geq T$.
\end{proof}

\begin{lem}
    \label{lem:order_remain_NAG}
    Suppose that $f$ is $\mu$-strongly convex, $x(t)$ satisfies \eqref{eq:NAG}, 
    and $r > 3$. Then,
    the following holds:
    \begin{equation}
        \frac{r^2 - 3r}{9t^2} \|x-x_*\|^2 = \mathrm{O} \left( t^{-\frac{2}{3}r} \right). 
        \label{eq:order_remain_NAG}
    \end{equation}
\end{lem}

\begin{proof}
    From inequality~\eqref{eq:non_negative_ineq}, 
    the expression of $\mathcal{E}(t)$ in~\eqref{eq:Lyap_NAG},
and Lemma \ref{lem:decrease_NAG},
    we have
    \begin{equation}
        \frac{1}{2} t^{\frac{2}{3}r} \left\| \dot{x} + \frac{2r}{3t} (x-x_*) \right\|^2 
        \leq \mathcal{E}(t)
        \leq \mathcal{E}(T)
        \label{eq:exp_norm_2_leq_E_T_NAG}
    \end{equation}    
    for $t \geq T$, 
    where $T$ is defined by \eqref{eq:def_T_NAG}.
    Here, we define
    \begin{equation}
    \label{eq:y_definition_NAG}
        y(t) = t^{\frac{2}{3}r} \left( x(t) - x_* \right).
    \end{equation}
    Then, from simple algebra, we have
    \begin{equation}
        \dot{x} + \frac{2r}{3t}(x(t) - x_* ) = \frac{\dot{y}(t)}{t^{\frac{2}{3}r}}.
        \label{eq:expr_by_dot_y_NAG}
    \end{equation}
    It follows from~\eqref{eq:exp_norm_2_leq_E_T_NAG} and~\eqref{eq:expr_by_dot_y_NAG} that
    \begin{equation}
        \|\dot{y}(t)\| \leq \sqrt{2\mathcal{E}(T)} t^{\frac{1}{3}r}.
    \end{equation}
By using this inequality, 
we can deduce that 
\begin{align}
    \| y(t) \| 
    &\leq 
    \| y(T) \| + \int_{T}^{t} \| \dot{y}(s) \| \mathrm{d}s
    \notag \\
    &  
    \leq 
    \| y(T) \| + \int_{T}^{t} \sqrt{2 \mathcal{E}(T)} \, s^{\frac{1}{3}r} \mathrm{d}s
    \notag \\
    & 
    \leq
    \| y(T) \| + \int_{0}^{t} \sqrt{2 \mathcal{E}(T)} \, s^{\frac{1}{3}r} \mathrm{d}s
    \notag \\
    & 
    =
    \| y(T) \| + \frac{3\sqrt{2 \mathcal{E}(T)}}{r+3} \, t^{\frac{1}{3}r+1},
\end{align}
which implies
    \begin{equation}
    \| y(t) \| t^{-\frac{2}{3}r}
    \leq
    \frac{3 \sqrt{2\mathcal{E}(T)}}{r+3} t^{-\frac{1}{3}r+1} 
    + 
    \|y(T)\| t^{-\frac{2}{3}r}. 
        \label{eq:key_ineq_for_norm_yt_NAG}
    \end{equation}
From~\eqref{eq:y_definition_NAG} and~\eqref{eq:key_ineq_for_norm_yt_NAG}, we have
\begin{align}
            \frac{r^2 - 3r}{9t^2} \|x-x_*\|^2 
            & = \frac{r^2 - 3r}{9t^2} \left( \| y(t) \| t^{-\frac{2}{3}r} \right)^2 
            \notag \\
            &\leq \frac{2r(r-3)\mathcal{E}(T)}{(r+3)^2} t^{-\frac{2}{3}r} 
            \notag \\
            &\phantom{\leq}+ \frac{2r(r-3) \sqrt{2\mathcal{E}(T)} \|y(T)\|}{3(r+3)} t^{-r-1} 
            \notag \\
            &\phantom{\leq}+ \frac{r(r-3) \|y(T)\|^2}{9} t^{-\frac{4}{3}r-2}.
        \label{eq:final_key_ineq_for_norm_NAG}
\end{align}
    Since $r>3$ and $\frac{2r(r-3)\mathcal{E}(T)}{(r+3)^2}$ is a constant,
        the first term in the right-hand side of~\eqref{eq:final_key_ineq_for_norm_NAG} is dominant. 
    Therefore we have~\eqref{eq:order_remain_NAG}. 
\end{proof}

\begin{thm}
    \label{thm:rate_NAG}
    Suppose that $f$ is $\mu$-strongly convex and $x(t)$ satisfies \eqref{eq:NAG}. Then,
    the following holds:
    \begin{equation}
        f - f_* = \mathrm{O} \left( t^{-\frac{2}{3}r} \right).
    \end{equation}
\end{thm}

\begin{proof}
This follows from Lemmas \ref{lem:order_main_NAG} and \ref{lem:order_remain_NAG}.
Note that the conclusion follows only from Lemma \ref{lem:order_main_NAG} if $0 < r \leq 3$.
\end{proof}

\section{Dynamical system $\ddot{x} + \frac{r}{t^{\alpha}}\dot{x} + \nabla f = 0$}

We consider the continuous dynamical system in~\eqref{eq:gen_NAG_alpha}. 
We assume $r>0$ and $0 < \alpha < 1$. 
This system was 
recently considered in \cite{ChengLiuShang2025}
as a generalized NAG system, 
and the given convergence rate is
\begin{equation}
    f - f_* = \mathrm{O} \left( \mathrm{e}^{-\frac{1}{2} \frac{r}{1-\alpha} t^{1-\alpha}} \right).
\end{equation}
Then, \cite{TabeiTanaka2025} proves an improved rate
\begin{equation}
    f - f_* = \mathrm{O} \left( \mathrm{e}^{-\left(\frac{2}{3} - \epsilon\right) \frac{r}{1-\alpha} t^{1-\alpha}} \right)
\end{equation}
with an arbitrary small positive constant $\epsilon$. 
In Theorem~\ref{thm:rate_alpha} below, 
we show a further improved rate. 

To this end, 
we perform our computer-assisted procedure in \cite{TabeiTanaka2025}
to list candidates of Lyapunov functions. 
Among them, 
we choose the function
\begin{align}
        \mathcal{E}(t) &= \mathrm{e}^{\frac{2}{3} \frac{r}{1-\alpha} t^{1-\alpha}} \biggl( f - f_* - \frac{r^2t^{-\alpha} - 3r\alpha t^{-1}}{9t^{\alpha}} \|x-x_*\|^2 \biggr) 
        \notag\\
        &\phantom{=} + \frac{1}{2} \mathrm{e}^{\frac{2}{3} \frac{r}{1-\alpha} t^{1-\alpha}} \left\|\dot{x} + \frac{2}{3} r t^{-\alpha} (x-x_*) \right\|^2.
        \label{eq:Lyap_gen_NAG_alpha}
\end{align}
Then we can show that $\dot{\mathcal{E}}(t) \leq 0$ for $t \geq T$ with
\begin{equation}
    T = \left( \frac{2r^2}{9\mu} \right)^{\frac{1}{2\alpha}}.
    \label{eq:def_T_alpha}
\end{equation}

\begin{lem}
    \label{lem:decrease_alpha}
    Suppose that $f$ is $\mu$-strongly convex and $x(t)$ satisfies 
    \eqref{eq:gen_NAG_alpha}. 
    Then $\mathcal{E}(t)$ in \eqref{eq:Lyap_gen_NAG_alpha} is non-increasing for $t \geq T$, 
    where $T$ is defined by \eqref{eq:def_T_alpha}. 
\end{lem}

\begin{proof}
    Because of \eqref{eq:gen_NAG_alpha}, simple algebra gives
\begin{align}
            \dot{\mathcal{E}}(t) &= -\mathrm{e}^{\frac{2}{3} \frac{r}{1-\alpha} t^{1-\alpha}} \biggl( \frac{2}{3}rt^{-\alpha} (f_* - f - \langle \nabla f, x_* - x \rangle) 
            \notag \\
            &\phantom{=} \ \ - \frac{2 r^3 t^{-2\alpha} - 9r\alpha (1+\alpha) t^{-2}}{27t^{\alpha}}\|x-x_*\|^2 \biggr).
\end{align}
    Since $f$ is $\mu$-strongly convex, the inequality
\begin{align}
            \dot{\mathcal{E}}(t) &\leq -\mathrm{e}^{\frac{2}{3} \frac{r}{1-\alpha} t^{1-\alpha}} 
            \notag \\
            &\phantom{=} \cdot \frac{ \left(9\mu t^{2\alpha} - 2r^2\right)r t^{-2\alpha} + 9r\alpha (1+\alpha) t^{-2}}{27t^{\alpha}}  \|x-x_*\|^2
\end{align}
    holds. Since $t^{2\alpha} \geq \frac{2r^2}{9\mu}$ holds for $t \geq T$, we have
    \begin{equation}
        \dot{\mathcal{E}}(t) \leq -\mathrm{e}^{\frac{2}{3} \frac{r}{1-\alpha} t^{1-\alpha}} \cdot \frac{r\alpha (1+\alpha)t^{-2}}{3t^{\alpha}} \|x-x_*\|^2 \leq 0.
    \end{equation}
    Therefore, $\mathcal{E}(t)$ is non-increasing for $t \geq T$. 
\end{proof}

\begin{lem}
    \label{lem:order_main_alpha}
    Suppose that $f$ is $\mu$-strongly convex and $x(t)$ satisfies \eqref{eq:gen_NAG_alpha}.  Then,
    the following holds:
    \begin{equation}
        f - f_* - \frac{r^2t^{-\alpha} - 3r\alpha t^{-1}}{9t^{\alpha}} \|x-x_*\|^2 = \mathrm{O} \left( \mathrm{e}^{-\frac{2}{3} \frac{r}{1-\alpha} t^{1-\alpha}} \right). 
        \label{eq:order_main_alpha}
    \end{equation}
\end{lem}

\begin{proof}
The left-hand side of \eqref{eq:order_main_alpha} is non-negative for $t \geq T$, 
where $T$ is defined by \eqref{eq:def_T_alpha}.
Indeed, 
the inequality $9t^{2\alpha} \geq \frac{2r^2}{\mu}$ holds for $t \geq T$. 
Then, 
since $f$ is $\mu$-strongly convex,
\begin{align}
            & f - f_* - \frac{r^2t^{-\alpha} - 3r\alpha t^{-1}}{9t^{\alpha}} \|x-x_*\|^2
            \notag \\
            &\geq \left(f - f_* - \frac{\mu}{2}\|x-x_*\|^2\right) + \frac{r\alpha t^{-1}}{3t^{\alpha}}\|x-x_*\|^2 \geq 0
            \label{eq:non_negative_ineq_alpha}
\end{align}
holds for $t \geq T$. 
    Then, by Lemma \ref{lem:decrease_alpha}, 
    we have \eqref{eq:order_main_alpha}
    from the same discussion as Lemma \ref{lem:Lyapunov_basic}
    for $t \geq T$, where $T$ is defined by \eqref{eq:def_T_alpha}.
\end{proof}

\begin{lem}
\label{lem:order_remain_alpha}
Suppose that $f$ is $\mu$-strongly convex and $x(t)$ satisfies \eqref{eq:gen_NAG_alpha}. 
Then, the following holds:
    \begin{equation}
        \frac{r^2t^{-\alpha} - 3r\alpha t^{-1}}{9t^{\alpha}} \|x-x_*\|^2 = \mathrm{O} \left( \mathrm{e}^{-\frac{2}{3} \frac{r}{1-\alpha} t^{1-\alpha}} \right).
        \label{eq:order_remain_alpha}
    \end{equation}
\end{lem}

\begin{proof}
Let $T'$ be defined by $T' = (3\alpha/r)^{1/(1-\alpha)}$
and suppose that $t \geq \max\{ T, T' \}$, 
where $T$ is defined by \eqref{eq:def_T_alpha}. 
    From inequality~\eqref{eq:non_negative_ineq_alpha}, 
    the expression of $\mathcal{E}(t)$ in~\eqref{eq:Lyap_gen_NAG_alpha}, 
    and Lemma \ref{lem:decrease_alpha}, 
    we have
    \begin{equation}
        \frac{1}{2} \mathrm{e}^{\frac{2}{3} \frac{r}{1-\alpha} t^{1-\alpha}} \left\|\dot{x} + \frac{2}{3} r t^{-\alpha} (x-x_*) \right\|^2 
        \leq \mathcal{E}(t)
        \leq \mathcal{E}(T).
        \label{eq:exp_norm_2_leq_E_T}
    \end{equation}
    Here, we define
    \begin{equation}
    \label{eq:y_definition_alpha}
        y(t) = \mathrm{e}^{\frac{2}{3}\frac{r}{1-\alpha}t^{1-\alpha}} \left( x(t) - x_* \right).
    \end{equation}
    Then, from simple algebra, we have
    \begin{equation}
        \dot{x} + \frac{2}{3} rt^{-\alpha} (x-x_*) = \frac{\dot{y}(t)}{\mathrm{e}^{\frac{2}{3} \frac{r}{1-\alpha}t^{1-\alpha}}}.
        \label{eq:expr_by_dot_y}
    \end{equation}
    It follows from~\eqref{eq:exp_norm_2_leq_E_T} and~\eqref{eq:expr_by_dot_y} that
    \begin{equation}
        \|\dot{y}(t)\| \leq \sqrt{2\mathcal{E}(T)} \mathrm{e}^{\frac{1}{3} \frac{r}{1-\alpha} t^{1-\alpha}}. 
    \end{equation}
By using this inequality, 
we can deduce that 
\begin{align}
       \|y(t)\| 
        & \leq 
        \|y(T)\| + \int_{T}^{t} \| \dot{y}(s) \| \mathrm{d}s 
        \notag \\
        & \leq 
        \|y(T)\| + \sqrt{2\mathcal{E}(T)} 
        \int_{T}^{t} \frac{3}{r} t^{\alpha} \frac{r}{3} s^{-\alpha} \mathrm{e}^{\frac{1}{3} \frac{r}{1-\alpha} s^{1-\alpha}} \mathrm{d}s 
        \notag \\
        & \leq 
        \|y(T)\| + \sqrt{2\mathcal{E}(T)} 
        \int_{0}^{t} \frac{3}{r} t^{\alpha} \frac{r}{3} s^{-\alpha} \mathrm{e}^{\frac{1}{3} \frac{r}{1-\alpha} s^{1-\alpha}} \mathrm{d}s  
        \notag \\
        & \leq 
        \|y(T)\| + \frac{3\sqrt{2\mathcal{E}(T)}}{r}t^{\alpha}\mathrm{e}^{\frac{1}{3} \frac{r}{1-\alpha}t^{1-\alpha}},
\end{align}
which implies
\begin{align}
& \|y(t)\| \mathrm{e}^{-\frac{2}{3}\frac{r}{1-\alpha}t^{1-\alpha}} 
\notag \\
& \leq 
\frac{3\sqrt{2\mathcal{E}(T)}}{r}t^{\alpha}\mathrm{e}^{- \frac{1}{3} \frac{r}{1-\alpha}t^{1-\alpha}} 
+ \|y(T)\|  \mathrm{e}^{-\frac{2}{3}\frac{r}{1-\alpha}t^{1-\alpha}}.
        \label{eq:key_ineq_for_norm_yt_alpha}
\end{align}
The inequality 
$r^2t^{-\alpha} - 3r\alpha t^{-1} \geq 0$
holds for $t \geq T'$. 
From 
this inequality, \eqref{eq:y_definition_alpha}, and
\eqref{eq:key_ineq_for_norm_yt_alpha}, 
we have 
\begin{align}
            &\frac{r^2t^{-\alpha} - 3r\alpha t^{-1}}{9t^{\alpha}} \|x-x_*\|^2 
            \notag \\
            & = \left( \frac{r^{2}}{9t^{2\alpha}} - \frac{r\alpha}{3t^{1+\alpha}} \right) \left( \| y(t) \| \mathrm{e}^{-\frac{2}{3}\frac{r}{1-\alpha}t^{1-\alpha}}  \right)^2 
            \notag \\
            &\leq 2\mathcal{E}(T) \mathrm{e}^{-\frac{2}{3} \frac{r}{1-\alpha}t^{1-\alpha}}  + \frac{2r\sqrt{2\mathcal{E}(T)} \|y(T)\|}{3t^{\alpha}} \mathrm{e}^{- \frac{r}{1-\alpha}t^{1-\alpha}} 
            \notag \\
            &\phantom{\leq}+  \frac{r^2 \|y(T)\|^2}{9t^{2\alpha}}  \mathrm{e}^{-\frac{4}{3} \frac{r}{1-\alpha}t^{1-\alpha}} - \frac{6\alpha\mathcal{E}(T)}{rt^{1-\alpha}}\mathrm{e}^{-\frac{2}{3} \frac{r}{1-\alpha}t^{1-\alpha}} 
            \notag \\
            &\phantom{\leq} - \frac{2\alpha\sqrt{2\mathcal{E}(T)} \|y(T)\|} {t} \mathrm{e}^{- \frac{r}{1-\alpha}t^{1-\alpha}} 
            \notag \\
            & \phantom{\leq} - \frac{r\alpha \|y(T)\|^2}{3t^{1+\alpha}} \mathrm{e}^{-\frac{4}{3} \frac{r}{1-\alpha}t^{1-\alpha}}.
        \label{eq:final_key_ineq_for_norm}
\end{align}
    Since $r>0$, $0<\alpha<1$ and $2\mathcal{E}(T)$ is a constant, 
    the first term in the right-hand side of~\eqref{eq:final_key_ineq_for_norm} is dominant. 
    Therefore we have 
    \eqref{eq:order_remain_alpha}. 
\end{proof}

\begin{thm}
    \label{thm:rate_alpha}
    Suppose that $f$ is $\mu$-strongly convex and $x(t)$ satisfies \eqref{eq:gen_NAG_alpha}. Then,
    the following holds:
    \begin{equation}
        f - f_* = \mathrm{O} \left( \mathrm{e}^{-\frac{2}{3} \frac{r}{1-\alpha} t^{1-\alpha}} \right). 
    \end{equation}
\end{thm}

\begin{proof}
This follows from Lemmas \ref{lem:order_main_alpha} and \ref{lem:order_remain_alpha}.
\end{proof}

\begin{rem}
The exponential factor 
\(
\mathrm{e}^{\frac{2}{3} \frac{r}{1 - \alpha} t^{1-\alpha}}
\)
in \eqref{eq:Lyap_gen_NAG_alpha}, 
\eqref{eq:exp_norm_2_leq_E_T}, and 
\eqref{eq:y_definition_alpha}
may equivalently be normalized as
\(
\mathrm{e}^{\frac{2}{3} \frac{r}{1 - \alpha} (t^{1-\alpha}-1)}
\), 
since this differs from the former factor
only by a positive constant independent of $t$.
This normalized factor converges pointwise to $t^{2r/3}$ as $\alpha \to 1$. 
Hence, 
the normalized decay factor in Theorem~\ref{thm:rate_alpha} converges pointwise to 
the polynomial decay factor $t^{-2r/3}$ 
appearing in Theorem~\ref{thm:rate_NAG}.
\end{rem}

\section{Conclusion}

We refined the evaluation of the Lyapunov functions obtained by our computer-assisted framework in \cite{TabeiTanaka2025}.
The key idea is to control the modified quantity in~\eqref{eq:mod_quant},
yielding sharper bounds.
As a result, we reproduced the known rate for the dynamical system in~\eqref{eq:NAG}
and derived a strictly improved rate for the system in~\eqref{eq:gen_NAG_alpha}.

A direction for future work is to automate this process,
incorporating both the discovery of Lyapunov functions and the analytical steps for deriving convergence rates.
This may lead to a systematic framework for analyzing and designing optimization-related dynamical systems. 
Also, our analysis relies on
the non-negativity of the quantity in \eqref{eq:mod_quant}, 
which requires the strong convexity of $f$. 
Relaxing this condition is another direction.

\acknowledgments

The authors thank Dr. Kansei Ushiyama for his valuable advice on this work. 
They thank the anonymous reviewer for their insightful comments.
The second author was supported by JSPS KAKENHI Grant Number
JP24K00536 (Grant-in-Aid for Scientific Research (B)).

\references

\end{document}